\documentclass[a4paper,11pt]{amsart}
\usepackage{amssymb,amscd}
\usepackage[alwaysadjust]{paralist}
\usepackage[capitalize]{cleveref}
\numberwithin{equation}{section}

\newtheorem{lem}[equation]{Lemma}

\newtheorem{thm}[equation]{Theorem}
\theoremstyle{definition}

\def\N{\mathbb N}
\def\R{\mathbb R}
\def\Q{\mathbb Q}
\def\Z{\mathbb Z}

\def\ve{\varepsilon}
\def\vf{\varphi}

\newcommand{\supp}{\operatorname{supp}}

\begin{document}


\title[Bottom of spectrum]
{On the bottom of spectra under coverings}
\author{Werner Ballmann}
\address
{WB: Max Planck Institute for Mathematics,
Vivatsgasse 7, 53111 Bonn and
Hausdorff Center for Mathematics,
Endenicher Allee 60, 53115 Bonn.}
\email{hwbllmnn\@@mpim-bonn.mpg.de}
\author{Henrik Matthiesen}
\address
{HM: Max Planck Institute for Mathematics,
Vivatsgasse 7, 53111 Bonn and
Hausdorff Center for Mathematics,
Endenicher Allee 60, 53115 Bonn.}
\email{hematt\@@mpim-bonn.mpg.de}
\author{Panagiotis Polymerakis}
\address{PP: Institut f{\"u}r Mathematik,
Humboldt-Universit\"at zu Berlin,
Unter den Linden 6, 10099 Berlin.}
\email{polymerp@hu-berlin.de}

\thanks{We would like to thank the Max Planck Institute for Mathematics
and the Hausdorff Center for Mathematics in Bonn for their support.}

\date{\today}


\begin{abstract}
For a Riemannian covering $M_1\to M_0$ of connected Riemannian manifolds
with respective fundamental groups $\Gamma_1\subseteq\Gamma_0$,
we show that the bottoms of the spectra of $M_0$ and $M_1$ coincide
if the right action of $\Gamma_0$ on $\Gamma_1\backslash\Gamma_0$ is amenable.
\end{abstract}

\maketitle

\section{Introduction}
\label{intro}

In this article, we study the behaviour under coverings
of the bottom of the spectrum of Schr\"odinger operators on Riemannian manifolds.

Let $M$ be a connected Riemannian manifold, not necessarily complete,
and $V\colon M\to\R$ be a smooth potential with associated \emph{Schr\"odinger operator} $\Delta+V$.
We consider $\Delta+V$ as an unbounded symmetric operator in the space $L^2(M)$ of square integrable functions on $M$ 
with domain $C^\infty_c(M)$, the space of smooth functions on $M$ with compact support.

For a non-vanishing Lipschitz continuous function on $M$ with compact support in $M$, we call
\begin{equation}\label{rayl}
  R(f) = \frac{\int_M ( |\nabla f|^2 + Vf^2 )}{\int_M f^2}
\end{equation}
the \emph{Rayleigh quotient of $f$}.
We let
\begin{equation}\label{botto}
   \lambda_0(M,V) = \inf R(f),
\end{equation}
where $f$ runs through all non-vanishing Lipschitz continuous functions on $M$ with compact support in $M$.
If $\lambda_0(M,V)>-\infty$, then $\Delta+V$ is bounded from below on $C^\infty_c(M)$
and $\lambda_0(M,V)$ is equal to the bottom of the spectrum of the \emph{Friedrichs extension of $\Delta+V$}.
If $\lambda_0(M,V)=-\infty$,
then the spectrum of any self-adjoint extension of $\Delta+V$ is not bounded from below.

Recall that $\Delta+V$ is essentially self-adjoint on $C^\infty_c(M)$ if $M$ is complete and $\inf V>-\infty$.
Then the unique self-adjoint extension of $\Delta+V$ is its closure.
In the case where $M$ is the interior of a complete Riemannian manifold $N$ with smooth boundary
and where $V$ extends smoothly to the boundary of $N$,
$\lambda_0(M,V)$ is equal to the bottom of the Dirichlet spectrum of $\Delta+V$ on $N$.

In the case of the \emph{Laplacian}, that is, $V=0$, we also write $\lambda_0(M)$
and call it the \emph{bottom of the spectrum of $M$}.
It is well known that $\lambda_0(M)$ is the supremum over all $\lambda\in\R$
such that there is a positive smooth $\lambda$-eigenfunction $f\colon M\to\R$
(see, e.g., \cite[Theorem 7]{CY}, \cite[Theorem 1]{FS}, or \cite[Theorem 2.1]{Su}.)
It is crucial that these eigenfunctions are not required to be square-integrable.
In fact, $\lambda_0(M)$ is exactly the border between the positive and the $L^2$ spectrum of $\Delta$
(see, e.g., \cite[Theorem 2.2]{Su}).

Suppose now that $M$ is simply connected
and let $\pi_0\colon M\to M_0$ and $\pi_1\colon M\to M_1$ be Riemannian subcovers of $M$.
Let $\Gamma_0$ and $\Gamma_1$ be the groups of covering transformations of $\pi_0$ and $\pi_1$, respectively, and assume that $\Gamma_1\subseteq\Gamma_0$.
Then the resulting Riemannian covering $\pi\colon M_1\to M_0$ satisfies $\pi\circ\pi_1=\pi_0$.
Let $V_0\colon M_0\to\R$ be a smooth potential and set $V_1=V_0\circ\pi$.

Since the lift of a positive $\lambda$-eigenfunction of $\Delta$ on $M_0$ to $M_1$ is a positive $\lambda$-eigenfunction of $\Delta$,
we always have $\lambda_0(M_0)\le\lambda_0(M_1)$ by the above characterization of the bottom of the spectrum of $\Delta$ by positive eigenfunctions.
In \cref{secpush},
we present a short and elementary proof of the inequality
which does not rely on the characterization of $\lambda_0$ by positive eigenfunctions:

\begin{thm}\label{down}
For any Riemannian covering $\pi\colon M_1\to M_0$ as above,
\begin{equation*}
  \lambda_0(M_0,V_0)\le\lambda_0(M_1,V_1).
\end{equation*}
\end{thm}

Brooks showed in \cite[Theorem 1]{Br2} that $\lambda_0(M_0)=\lambda_0(M_1)$ in the case where $M_0$ is complete, has \emph{finite topological type},
and $\pi$ is \emph{normal} with \emph{amenable} group $\Gamma_1\backslash\Gamma_0$ of covering transformations.
B\'erard and Castillon extended this in \cite[Theorem 1.1]{BC} to
$\lambda_0(M_0,V_0)=\lambda_0(M_1,V_1)$
in the case where $M_0$ is complete,
$\pi_1(M_0)$ is finitely generated (this assumption occurs in point (1) of their Section 3.1),
and the right action of $\Gamma_0$ on $\Gamma_1\backslash\Gamma_0$ is amenable.
We generalize these results as follows:

\begin{thm}\label{up}
If the right action of $\Gamma_0$ on $\Gamma_1\backslash\Gamma_0$ is amenable,
then \[\lambda_0(M_0,V_0)=\lambda_0(M_1,V_1).\]
\end{thm}

Here a right action of a countable group $\Gamma$ on a countable set $X$ is said to be \emph{amenable} if there exists a $\Gamma$-invariant mean on $L^\infty(X)$.
This holds if and only if the action satisfies the \emph{F\o{}lner condition}:
For any finite subset $G\subseteq\Gamma$ and $\ve>0$,
there exists a non-empty, finite subset $F\subseteq X$, a \emph{F\o{}lner set}, such that
\begin{equation}\label{folner}
  |F\setminus Fg| \le \ve |F|
\end{equation}
for all $g\in G$.
By definition, $\Gamma$ is \emph{amenable} if the right action of $\Gamma$ on itself is amenable,
and then any action of $\Gamma$ is amenable.
 
In comparison with the results of Brooks, B\'erard, and Castillon,
the main point of \cref{up} is that we do not need any assumptions on metric and topology of $M_0$.
A main new point of our arguments is that we adopt our constructions
more carefully to the different competitors for $\lambda_0$ separately.

\section{Fundamental domains and partitions of unity}
\label{secfun}

Choose a complete Riemannian metric $h$ on $M_0$.
In what follows, geodesics, distances, and metric balls in $M_0$, $M_1$, and $M$
are taken with respect to $h$ and its lifts to $M_1$ and $M$, respectively.

Fix a point $x$ in $M_0$.
For any $y\in\pi^{-1}(x)$, let
\begin{equation}
  D_y = \{ z\in M_1 \mid \text{$d(z,y)\le d(z,y')$ for all $y'\in\pi^{-1}(x)$} \}
\end{equation}
be the  \emph{fundamental domain} of $\pi$ centered at $y$.
Then $D_y$ is closed in $M_1$,
the boundary $\partial D_y$ of $D_y$ has measure zero in $M_1$,
and $\pi\colon D_y\setminus\partial D_y\to M_0\setminus C$ is an isometry,
where $C$ is a subset of the cut locus ${\rm Cut}(x)$ of $x$ in $M_0$.
Recall that ${\rm Cut}(x)$ is of measure zero.
Moreover, $M_1=\cup_{y\in\pi^{-1}(x)}D_y$, $y\in\pi^{-1}(x)$.

\begin{lem}\label{number}
For any $\rho>0$,
there is an integer $N(\rho)$ such that any $z$ in $M_1$ is contained
in at most $N(\rho)$ metric balls $B(y,\rho)$, $y\in\pi^{-1}(x)$.
\end{lem}

\begin{proof}
Let $z\in B(z_1,\rho)\cap B(y_2,\rho)$ with $y_1\ne y_2$ in $\pi^{-1}(x)$
and $\gamma_1,\gamma_2\colon[0,1]\to M_1$
be minimal geodesics from $y_1$ to $z$ and $y_2$ to $z$, respectively.
Then $\sigma_1=\pi\circ\gamma_1$ and $\sigma_2=\pi\circ\gamma_2$
are geodesic segments form $x$ to $\pi(z)$.
Since $y_1\ne y_2$, $\sigma_1$ and $\sigma_2$ are not homotopic relative to $\{0,1\}$.
Hence, if $z$ lies in in the intersection of $n$ pairwise different balls $B(y_i,\rho)$
with $y_1,\dots,y_n\in\pi^{-1}(x)$,
then the concatenations $\sigma_1^{-1}*\sigma_i$ represent $n$
pairwise different homotopy classes of loops at $x$ of length at most $2\rho$.
Hence $n$ is at most equal to the number $N(\rho)$ of homotopy classes of loops
at $x$ with representatives of length at most $2\rho$.
\end{proof}

\begin{lem}\label{coco}
If $K\subseteq M_0$ is compact, then $\pi^{-1}(K)\cap D_y$ is compact.
More precisely, if $K\subseteq B(x,r)$, then $\pi^{-1}(K)\cap D_y\subseteq B(y,r)$.
\end{lem}

\begin{proof}
Choose $r>0$ such that $K\subseteq B(x,r)$.
Let $z\in\pi^{-1}(K)\cap D_y$ and $\gamma_0$ be a minimal geodesic from $\pi(z)\in K$ to $x$.
Let $\gamma$ be the lift of $\gamma_0$ to $M_1$ starting in $z$.
Then $\gamma$ is a minimal geodesic from $z$ to some point $y'\in\pi^{-1}(x)$.
Since $z\in D_y$, this implies 
\begin{equation*}
  d(z,y) \le d(z,y') \le L(\gamma) = L(\gamma_0) < r.
\end{equation*}
Hence $\pi^{-1}(K)\cap D_y\subseteq B(y,r)$.
\end{proof}

Let $K\subseteq M_0$ be a compact subset and choose $r>0$ such that $K\subseteq B(x,r)$.
Let $\psi\colon\R\to\R$ be the function which is equal to $1$ on $(-\infty,r]$, to $t+1-r$ for $r\le t\le r+1$, and to $0$ on $[r+1,\infty]$.
For $y\in\pi^{-1}(x)$, let $\psi_y=\psi_y(z)=\psi(d(z,y))$.
Note that $\psi_y=1$ on $\pi^{-1}(K)\cap D_y$ and that $\supp\psi_y=\bar B(y,r+1)$.

\begin{lem}\label{number2}
Any $z$ in $M_1$ is contained in the support of at most $N(r+1)$
of the functions $\psi_y$, $y\in\pi^{-1}(x)$.
\end{lem}

\begin{proof}
This is clear from \cref{number}
since $\supp\psi_y$ is contained in the ball $B(y,r+1)$.
\end{proof}

In particular,
each point of $M_1$ lies in the support of only finitely many of the functions $\psi_y$. 
Therefore the function $\psi_1=\max\{1-\sum\psi_y,0\}$ is well defined.
By \cref{coco}, we have $\supp\psi_1\cap\pi^{-1}(K)=\emptyset$.
Together with $\psi_1$, the functions $\psi_y$ lead to a partition of unity on $M_1$ with functions $\vf_1$ and $\vf_y$, $y\in\pi^{-1}(x)$, given by
\begin{equation}\label{partu}
  \vf_1 = \frac{\psi_1}{\psi_1+\sum_{z\in\pi^{-1}(x)}\psi_z}
  \quad\text{and}\quad
  \vf_y = \frac{\psi_y}{\psi_1+\sum_{z\in\pi^{-1}(x)}\psi_{z}}.
\end{equation}
Note that $\supp\vf_1=\supp\psi_1$ and $\supp\vf_y=\supp\psi_y$ for all $y\in\pi^{-1}(x)$.

\begin{lem}\label{unilip}
The functions $\vf_y$, $y\in\pi^{-1}(x)$,
are Lipschitz continuous with Lipschitz constant $3N(r+1)$.
\end{lem}

\begin{proof}
The functions $\psi_y$, $y\in\pi^{-1}(x)$, are Lipschitz continuous with Lipschitz constant $1$
and take values in $[0,1]$.
Hence $\psi_1$ is Lipschitz continuous with Lipschitz constant $N=N(r+1)$, by \cref{number2},
and takes values in $[0,1]$.
Therefore the denominator $\chi=\psi_1+\sum_{z\in\pi^{-1}(x)}\psi_{z}$ in the fraction defining the $\vf_y$ is Lipschitz continuous and takes values in $[1,N]$.
Hence
\begin{align*}
  |\vf_y(z_1) - \vf_y(z_2)|
  &\le  \frac{|(\chi(z_2)-\chi(z_1))\psi_y(z_1)
  + \chi(z_1)(\psi_y(z_1)-\psi_y(z_2))|}{\chi(z_1)\chi(z_2)} \\
  &\le \frac{(2N+N) d(z_1,z_2)}{\chi(z_1)\chi(z_2)}
  \le 3N d(z_1,z_2).
  \qedhere
\end{align*}
\end{proof}

As a consequence of \cref{unilip},
we get that $\vf_1=1-\sum\vf_y$ is also Lipschitz continuous with Lipschitz constant $6N(r+1)^2$.

\section{Pulling up}
\label{secpull}

Let $f$ be a non-vanishing Lipschitz continuous function on $M_0$ with compact support
and let $f_1=f\circ\pi$.
We will construct a cutoff function $\chi$ on $M_1$ such that $R(\chi f_1)$ is close to $R(f)$.

Let $g$ be the given Riemannian metric on $M$
and $h$ be a complete background Riemannian metric on $M$ as in \cref{secfun}.
Then there is a constant $A\ge1$ such that
\begin{equation}\label{a}
  A^{-1}g \le h \le Ag
\end{equation}
on the support of $f$.
We continue to take distances and metric balls in $M_0$, $M_1$, and $M$
with respect to $h$ and its respective lifts to $M_1$ and $M$.

Fix a point $x$ in $M_0$.
With $K=\supp f$ and $r>0$ such that $K\subseteq B(x,r)$,
we get a partition of unity with functions $\vf_1$ and $\vf_y$, $y\in\pi^{-1}(x)$, as above.

Fix preimages $u\in M$ and $y=\pi_1(u)\in M_1$ of $x$ under $\pi_0$ and $\pi$, respectively.
Write $\pi_0^{-1}(x)=\Gamma_0u$ as the union of $\Gamma_1$-orbits $\Gamma_1gu$,
where $g$ runs through a set $R$ of representatives of the right cosets of $\Gamma_1$ in $\Gamma_0$, that is, of the elements of $\Gamma_1\backslash\Gamma_0$.
Then $\pi^{-1}(x)=\{\pi_1(gu)\mid g\in R\}$.
Let
\begin{equation*}
\begin{split}
  S &= \{ s\in R \mid d(y,\pi_1(su)) \le 2r+2 \} \\
  &= \{ s\in R \mid \text{$d(u,tsu)) \le 2r+2$ for some $t\in\Gamma_1$}\}, \\
  T &= \{ t \in \Gamma_1 \mid \text{$d(u,t gu)\le2r+2$ for some $s\in S$}\}, \\
  G &= TS \subseteq \Gamma_0.
\end{split}
\end{equation*}
Since the fibres of $\pi$ and $\pi_0$ are discrete,
$S$ and $T$ are finite subsets of $\Gamma_0$, hence also $G$.

Let $\ve>0$ and $F\subseteq\Gamma_1\backslash\Gamma_0$ be a F\o{}lner set
for $G$ and $\ve$ satisfying \eqref{folner}.
Let
\begin{equation*}
  P = \{ g \in R \mid \Gamma_1g \in F \} \subseteq R
\end{equation*}
and set
\begin{equation*}
  \chi = \sum_{g\in P}\vf_{\pi_1(gu)}.
\end{equation*}
Since $|P|=|F|<\infty$, $\supp\chi$ is compact.
Hence, by \cref{unilip},
$\chi f_1$ is compactly supported and Lipschitz continuous on $M_1$.
Let
\begin{equation*}
  Q = \{ y \in \pi^{-1}(x) \mid \text{$(\chi f_1)(z)\ne0$ for some $z\in D_y$} \}.
\end{equation*}
To estimate the Rayleigh quotient of $\chi f_1$,
it suffices to consider $\chi f_1$ on the union of the $D_y$, $y\in Q$.
We first observe that
\begin{equation*}
  P_1 = \{ \pi_1(gu) \mid g \in P \} \subseteq Q.
\end{equation*}
To show this, let $y=\pi_1(gu)$ and observe that $f_1$ does not vanish identically
 on $\pi^{-1}(K)\cap D_y$ and that $\vf_y$ is positive on $\pi^{-1}(K)\cap D_y$.
Since $R$ is a set of representatives of the right cosets of $\Gamma_1$ in $\Gamma_0$,
there exists a one-to-one correspondence between $P$ and $P_1$, and hence
\begin{equation*}
  |P| = |P_1| \le |Q|. 
\end{equation*}
The problematic subset of $Q$ is
\begin{equation*}
  Q_- = \{ y \in Q \mid \text{$0<\chi(z)<1$ for some $z\in\pi^{-1}(K)\cap D_y$}\}.
\end{equation*}
Let now $y\in Q_-$ and $z\in\pi^{-1}(K)\cap D_y$ with $0<\chi(z)<1$.
Since $\pi_1(gu)$, $g\in R$, runs through all points of $\pi^{-1}(x)$,
we have $\sum_{g\in R} \vf_{\pi_1(gu)}(z)=1$.
Hence there are $g_1,\dots,g_k\in R\setminus P$ such that  $\vf_{\pi_1(g_iu)}(z)\ne0$ and
\begin{equation*}
  \chi(z) + \sum \vf_{\pi_1(g_iu)}(z) =1.
\end{equation*}
Furthermore, there has to be a $g\in P$ with $\vf_{\pi_1(gu)}(z)\ne0$.
Then the supports of the functions $\vf_{\pi_1(gu)}$ and $\vf_{\pi_1(g_iu)}$ intersect
and we get $d(\pi_1(gu),\pi_1(g_iu))\le2r+2$.
That is, we have $d(gu,h_ig_iu)\le2r+2$ for some $h_i\in\Gamma_1$.
We conclude that
\begin{align*}
  d(u,g^{-1}h_ig_iu)) = d(gu,h_ig_iu) \le 2r+2.
\end{align*}
Since $\pi_1$ is distance non-increasing,
we get that there are $s_i\in S$ and $t_i\in T$ such that $g^{-1}h_ig_i=t_is_i$,
and then $h_ig_i=gt_is_i$.
Since $g_i\notin P$,
we conclude that $\Gamma_1gt_is_i\notin F$, i.e., $\Gamma_1 g \in F \setminus F(t_is_i)^{-1}$.
Since $(t_is_i)^{-1}\in G$, there are at most $\ve|F||G|$ such elements $g\in P$.
Since $d(y,z)\le r$ and $d(z,\pi_1(gu))\le r+1$,
we conclude with \cref{number} that for fixed $g \in P$ there are at most $N(2r+1)$ such $y\in Q$.
We conclude that
\begin{equation}\label{qmiq}
\begin{split}
  |Q_-| &\le \ve |F||G|N(2r+1) \\
  &= \ve |P||G|N(2r+1) \le \ve |Q||G|N(2r+1).
\end{split}
\end{equation}
We now estimate the Rayleigh quotient of $\chi f_1$.
For any $y\in Q_+=Q\setminus Q_-$,
we have $\chi=1$ on $\pi^{-1}(K)\cap D_y$ and therefore
\begin{equation*}
\begin{split}
  \int_{D_y} \{|\nabla(\chi f_1)|^2 + V_1(\chi f_1)^2\}
  &= \int_{D_y} \{|\nabla f_1|^2 + V_1f_1^2\} \\
  &= \int_{M_0} \{ |\nabla f|^2 + Vf^2 \}
\end{split}
\end{equation*}
and
\begin{equation*}
  \int_{D_y} \chi^2f_1^2 = \int_{D_y} f_1^2 = \int_{M_0} f^2,
\end{equation*}
where, here and below,
integrals, gradients, and norms are taken with respect to the original Riemannian metric $g$ on $M$.

For any $y\in Q_-$, we have
\begin{equation*}
  \int_{D_y} \chi^2f_1^2 \le \int_{M_0} f^2
  \quad\text{and}\quad
  \int_{D_y} |V_1|\chi^2f_1^2 \le C_0 \int_{M_0} f^2,
\end{equation*}
where $C_0$ is the maximum of $|V_0|$ on $\supp f=K$.
By \cref{number2}, \cref{unilip}, and \eqref{a},
we have $|\nabla\chi|^2\le 9N(r+1)^4A$.
Therefore
\begin{align*}
  \int_{D_y} |\nabla(\chi f_1)|^2
  &\le 2\int_{D_y}\{|\nabla\chi|^2 f^2 + \chi^2|\nabla f\circ\pi|^2|\} \\
  &\le 18N(r+1)^4A \int_{M_0} f^2 + 2\int_{M_0} |\nabla f|^2.
\end{align*}
In conclusion,
\begin{equation*}
  \int_{D_y} \{ |\nabla(\chi f_1)|^2 + |V_1|\chi^2f_1^2\}  \le C
\end{equation*}
for any $y\in Q_-$,
where $C>0$ is an appropriate constant, which depends on $f$,
but not on $y$ or the choice of $\ve$ and $F$.
With $D=|G|N(2r+1)$, we obtain from \eqref{qmiq} that
\begin{equation*}
  |Q_-| \le \frac{\ve D}{1-\ve D}|Q_+|,
\end{equation*}
and conclude that
\begin{align*}
  R(\chi f_1)
  &= \frac{\int \{ |\nabla(\chi  f_1)|^2 + V_1\chi^2f_1^2 \}}{\int (\chi f_1)^2} \\
  &= \frac{\sum_{y\in Q} \int_{D_y} \{ |\nabla f_1|^2 + V_1f_1^2 \}}
  {\sum_{y\in Q} \int_{D_y} f_1^2} \\
  &\le \frac{\sum_{y\in Q_+} \int_{D_y} \{ |\nabla f_1|^2 + V_1f_1^2 \}
  + \ve CD |Q_+|/(1-\ve D) }{\sum_{y\in Q_+} \int_{D_y} f_1^2} \\
  &= \frac{\int_{M_0} \{ |\nabla f|^2 + Vf^2 \} + \ve CD/(1-\ve D) }{\int_{M_0} f^2} \\
  &= R(f) + \frac{\ve CD}{(1-\ve D)\int_{M_0} f^2}.
\end{align*}
For $\ve\to0$, the right hand side converges to $R(f)$.

\begin{proof}[Proof of \cref{up}]
By \cref{down}, we have $\lambda_0(M_0,V_0)\le\lambda_0(M_1,V_1)$.
By \eqref{rayl}, the bottom of the spectrum of Schr\"odinger operators is given by the infimum of corresponding Rayleigh quotients $R(f)$ of Lipschitz continuous functions with compact support.
The arguments above show that, for any such function $f$ on $M_0$ and any $\delta>0$,
there is a Lipschitz continuous function $\chi f_1$ on $M_1$ with compact support and Rayleigh quotient at most $R(f)+\delta$.
Therefore we also have $\lambda_0(M_0,V_0)\ge\lambda_0(M_1,V_1)$.
\end{proof}

\section{Pushing down}
\label{secpush}

Let $f$ be a Lipschitz continuous function on $M_1$ with compact support.
Define the \emph{push down} $f_0\colon M_0\to\R$ of $f$ by
\begin{equation*}
  f_0(x) = \big( \sum_{y\in\pi^{-1}(x)}f(y)^2 \big)^{1/2}.
\end{equation*}
Since $\supp f$ is compact, the sum on the right hand side is finite for all $x\in M_0$,
and hence $f_0$ is well defined.
We have $\supp f_0=\pi(\supp f)$, and hence $\supp f_0$ is compact.
Furthermore, $f_0$ is differentiable at each point $x$, where $f$ is differentiable at all $y\in\pi^{-1}(x)$ and $f(y)\ne0$ for some $y\in\pi^{-1}(x)$,  and then
\begin{equation*}
  \nabla f_0(x) = \frac{1}{f_0(x)}\sum_{y\in\pi^{-1}(x)}f(y)\pi_*(\nabla f(y)).
\end{equation*}
For the norm of the differential of $f_0$ at $x$, we get
\begin{align*}
  |\nabla f_0(x)|^2
  &\le \frac{1}{f_0(x)^2}\bigg|\sum_{y\in\pi^{-1}(x)}f(y)\pi_*(\nabla f(y)) \bigg|^2\\
  &\le \frac{1}{f_0(x)^2} \sum_{y\in\pi^{-1}(x)}f(y)^2\sum_{y\in\pi^{-1}(x)}|\nabla f(y)|^2 \\
  &= \sum_{y\in\pi^{-1}(x)}|\nabla f(y)|^2.
\end{align*}
Furthermore, $f_0$ is differentiable with vanishing differential at almost any point of $\{f_0=0\}$.  
Therefore $f_0$ is Lipschitz continuous and
\begin{equation*}
  \int_{M_0} f_0^2 = \int_{M_1} f^2,
  \quad
  \int_{M_0} V_0f_0^2 = \int_{M_1} V_1f^2,
  \quad
  \int_{M_0} |\nabla f_0|^2 \le \int_{M_1} |\nabla f|^2.
\end{equation*}
In particular, we have $R(f_0)\le R(f)$.

\begin{proof}[Proof of \cref{down}]
For any non-vanishing Lipschitz continuous function $f$ on $M_1$ with compact support,
the push down $f_0$ as above is a Lipschitz continuous function on $M_0$ with compact support and Rayleigh quotient $R(f_0)\le R(f)$.
The asserted inequality follows now from the characterization of the bottom of the spectrum by Rayleigh quotients as in \eqref{rayl}.
\end{proof}

\section{Final remarks}

It is well-known that any countable group is the fundamental group of a smooth $4$-manifold.
(A variant of the usual argument for finitely presented groups,
taking connected sums of $S^1 \times S^3$ and performing surgeries,
can be used to produce $5$-manifolds with fundamental group any countable group.)
In particular, for a non-finitely generated, amenable group $G$,
e.g., $G=\bigoplus_{n \in \N} \Z$ or $G=\Q$,
there is a smooth manifold $M$ with $\pi_1(M)\cong G.$
In contrast to the results in \cite{BC, Br2}, our main result also applies to such examples.

Moreover, we do not assume $\lambda_0(M_0,V_0)>-\infty.$
Given any non-compact manifold $M_0$, it is indeed easy to construct a smooth potential $V_0$
such that $\lambda_0(M_0,V_0)=-\infty.$
In fact, it suffices that $V_0(x)$ tends to $-\infty$ sufficiently fast as $x\to\infty$.



\end{document}